\documentclass[preprint,authoryear,12pt]{elsarticle}

\usepackage[T1]{fontenc} 
\usepackage[utf8]{inputenc} 
\usepackage[frenchb,english]{babel} 
\usepackage{fourier}
\usepackage[scaled=0.875]{helvet}
\usepackage{courier}
\usepackage{dsfont}

\usepackage{amsmath,amsthm,amssymb,amsfonts,epic,latexsym,
hyperref,natbib}

\usepackage{graphicx}
\journal{ $\heartsuit$}

\usepackage{fourier}
\usepackage[scaled=0.875]{helvet}
\usepackage{courier}
\usepackage{dsfont}
\usepackage[usenames]{color}

\newtheorem{theo}{Theorem}[section]

\newtheorem{prop}[theo]{Proposition}

\newtheorem{assu}{Assumption}

\newtheorem{lemm}[theo]{Lemma}
\newtheorem{rema}[theo]{Remark}

\def\R{\mathbb{R}}
\def \N{\mathbb{N}}
\def \Z{\mathbb{Z}}
\def \L{\mathbb{L}}

\def\X{\mathbf{X}}
\def\LL{\mathcal{L}}

\def\w{\omega}
\def\E{\mathbb{E}} 

\def\P{\mathbb{P}} 

\def\PP{\mathbf{P}} 
\def\EE{\mathbf {E}} 

\def\IS{\mathbb{Q}} 

\def\1{\mathds{1}} 
\newcommand{\Var}{\mathbb{V}\mbox{ar}}

\def\l{\ell}
\def \eps{\varepsilon} 

\def\FF{\mathcal{F}}

\def\to{\rightarrow}
\def\8{\infty}
\def\cL{\mathcal{L}}
\def\M{\mathcal {M}}

\def\dd{\mathrm{d}} 
\def\to{\rightarrow}

\def\QQ{\mathbb{Q}}

\def\1{\mathrm{1}}

\def\X{\mathrm{X}}

\def\nn{\nonumber}

\def\GG{\mathcal{G}}

\def\BB{\mathcal{B}}
\def\FF{\mathcal{F}}
\def\GG{\mathcal{G}}

\def\KK{\mathcal{K}}
\def\S{\mathcal{S}}
\def\C{\mathcal{C}}

\newcommand{\ec}{\color{black}}

\begin{document}

\begin{frontmatter}

\title{ Limit of the environment viewed from Sinaï's walk
 }


\author[m1]{Francis Comets 
}
\ead{comets@lpsm.paris}
\author[m2,m3]{Oleg Loukianov}
\ead{oleg.loukianov@u-pec.fr}
\author[m2]{Dasha Loukianova }
\ead{dasha.loukianova@univ-evry.fr}

\address[m1]{Laboratoire de Probabilit\'es, Statistique et  Mod\'elisation,
  Universit\'e Paris Diderot, UMR~CNRS~8001,\\ 75205 Paris cedex 13, France.} 
\address[m2]{Laboratoire de Mathématiques et Mod\'elisation d'\'Evry, Universit\'e d'\'Evry Val d'Essonne, UMR~CNRS~8071, USC~INRA,\\ 23 Boulevard de France 91037 Evry cedex, France.}
\address[m3]{D\'epartement Informatique, IUT de Fontainebleau, Universit\'e Paris Est.}

\cortext[cor1]{Corresponding author: Dasha Loukianova }
\begin {abstract}For Sinaï's walk $(X_k)$ we show that the empirical measure of the environment seen from the particle $(\bar\w_k)$ converges in law to some  random measure $\S_\8. $ This limit measure is
explicitly given in terms of the infinite valley, which construction goes back to \cite {Golosov}. 
As a consequence an "in law" ergodic theorem holds:
 \begin{equation*}
\frac 1n\sum_{k=1}^nF (\bar{\omega}_k) \stackrel{\LL}{\longrightarrow} \int_\Omega F d\S_\8 \, .
\end{equation*} 
When the last limit is deterministic, it holds in probability. This allows some extensions to the recurrent case of the ballistic "environment's method" dating back to  \cite{KozMol}. In particular, we show an LLN and a mixed CLT for the sums $\sum_{k=1}^nf(\Delta X_k)$ where $f$ is bounded and 
depending on the steps $\Delta X_k:=X_{k+1}-X_k$.
\end{abstract}


\end{frontmatter}

\section{Introduction, assumptions and main results} \label{sec:intro}

 \subsection{Model}
Let $\w=\{\w(x);\ {x\in\Z}\}$ 
be a collection of i.i.d. random variables taking values in $[0,1]$.  Denote $\Omega:=[0,1]^{\Z}$,  $\P$ the distribution of $\w$ on $(\Omega,\BB(\Omega))$ and  $\E$ the expectation under this law.  For fixed $\w\in\Omega,$ let $X=(X_k)_{k\in\N},$ be the time-homogeneous   Markov chain on $\Z_+$ with transition 
function $ p^{\w}$ given 
 by
$p^{\w}(0,1)=1$, and for all $x\in\Z_+^*,$
\[  p^{\w}(x,y)=\left  \{\begin{array}{lr} \w(x)&\mbox{if}\
y=x+1,\\ 1-\w(x)&\mbox{if}\ y=x-1,\\ 0&\mbox{otherwise}.
\end{array}\right.
\]
For $x\in\Z_+,$ and fixed $\w\in\Omega$ we denote  by $P^{\w}_x$  the law on $(\Z_+^{\N}, \BB(\Z_+^{\N}))$ of the Markov chain $X$ starting from $x$. This is the \emph{quenched } law of $X$.
The law of the couple $(\w,X)$ is the probability measure $\PP_x$ on $(\Omega\times \Z_+^{\N}
,\BB(\Omega)\otimes \BB(\Z_+^{\N})) $ 
defined for all $x\in\Z_+$ and all $F\in\BB(\Omega)$ and $G\in \BB(\Z_+^{\N})$ by:

\[ 
\PP(F\times G)=\int_F P_x^{\omega}(G)\P(\dd\omega),
\]
this  is the  \emph{annealed}  law. 
 The annealed law is also dependent on the starting point of the walk, but this dependence is less important, because the walk is not 
 a Markov chain under this law. We do not  keep this dependence in our notation.  
We  write  $E_x^{\w}$ and
$\EE$   for   the   corresponding   quenched   and   annealed
expectations, respectively.
For simplicity and following \cite {Golosov} and \cite {GPS} we consider the walk on the positive integers reflected at $0,$ but we need the environnement be defined on $\Z$ to define later the infinite valley of the potential. 
  Denote, for $x\in\Z,$  
  \begin{equation}\label{eq:rho}
  \rho_x = \frac {1-\w(x)}{\w (x)}.
  \end{equation}
It was shown by \cite{Sol} that when 
\begin{equation}\label{rec}
\E\log \rho_0=0,
\end{equation}
 for $\P$-almost all $\w$ the Markov chain $X$ is recurrent, otherwise the walk is     
transient.
This paper focuses on the recurrent case, hence \eqref {rec} will be in force for all our results. 

\subsection{ Motivation of the paper: environment viewed from the particle }

For $\omega \in \Omega$ and $x \in \Z$, denote by  $T_x$ the shift operator $T_x: \Omega \to \Omega$, 
which shifts the environment by the vector $x$,
 i.e. 
 $$ \forall y\in\Z,\quad  (T_x \omega)(y)=\omega(x+y).$$
The environment seen from the walker is the $\Omega$-valued process $(\bar{\omega}_k)$ given by:
$$   \bar{\omega}_k= T_{X_k} \omega \, ,\quad  k\in\N\,.$$

It is well known since \cite{KozMol} that $( \bar{\omega}_k, k \geq 0)$ is a Markov chain (with respect to both $\PP$ and $P_0^{\w}$  ), 
with the transition kernel
\begin{equation} \label{eq:defQ}
R(\w, d\w')=\w(0)\delta _{T_1 \w}(d \w ')+(1-\w(0))\delta_{T_{-1}\w}(d\w ').
\end{equation}
The state space of this Markov chain is very complex, however, in the transient ballistic case, which is characterised (\cite{Sol}) by the linear speed  of escape of the  walk  to infinity:   
\begin{equation}\label{eq:asvel}
X_n/n\to v\neq 0,
\end{equation}
\cite{KozMol} showed that there exists a unique invariant probability $\QQ$ for the kernel $T$, which is absolutely continuous with respect to $\P$, with an explicit density $f=d\QQ/d\P$
(see \cite{Molchanov-StFlour} p.273 or Theorem 1.2 in \cite{Sznitman-10lectures}).  
In particular, Birkhoff's a.s. ergodic theorem applies to additive functionals of $(\bar{\omega}_k)$ and gives for all $F:\Omega\to\R,$ s.t. $\E [|F|\times f]<\infty,$ 
 \begin{equation}\label{a.s.}
\frac 1n\sum_{k=1}^nF (\bar{\omega}_k){\longrightarrow}
 \int_\Omega F(\w)f(\w)\P(d\w)  \quad \PP-a.s. .
\end{equation}

 This constitutes the basis of the "method of the environnement viewed from the particle". To recall it briefly, let us sketch the proof of Solomon's result \eqref{eq:asvel} on the asymptotic velocity for the ballistic random walk. Let $\Delta X_n:=X_{n+1}-X_n$, 
$\FF_n=\sigma\{\Delta X_0,\ldots, \Delta X_n,\;  \w(X_0),\ldots, \w(X_{n+1})\}.$ 
We can write the classical martingale differences decomposition:
\begin{equation}\label{introdecomp}
 {X_n}/n= 1/n\sum_{k=1}^n[\Delta X_k-\EE(\Delta X_k|\FF_{k-1})]+1/n\sum_{k=1}^n\EE(\Delta X_k|\FF_{k-1}).
\end{equation}The first sum in \eqref{introdecomp} is composed of centred, uncorrelated terms ($k$-th term is  $\FF_k$ measurable ). This first sum tends to zero in $\L^2$ and, using the martingale's convergence, even a.s..
Moreover, since $$\EE(\Delta X_k|\FF_{k-1})=\w(X_k)-1(1-\w(X_k))=2\bar\w_k(0)-1, $$ for the second term of \eqref{introdecomp} we can apply Brirkhoff's theorem and using the explicit expression of $f$ \cite {Molchanov-StFlour} p.273 get: 
 $$ 1/n\sum_{k=1}^n\EE(\Delta X_k|\FF_{k-1})=1/n\sum_{k=1}^n(2\bar\w_k(0)-1)\longrightarrow \int_\Omega (2\w(0)-1)f(\w)\P(d\w) =v\quad a.s.,$$
therefore $X_n/n\to v\quad a.s.$.
 For further illustration of this method see
 \cite{Sznitman-10lectures},
 \cite{ZeitouniSF} and \cite{Bog}.
In this work we are also interested in the limits of additive functionals $\frac 1n\sum_{k=1}^nF (\overline{\omega}_k).$ Knowing such limits allows to extend the environnement's method to the recurrent case. Besides this theoretical motivation, such additive functionals arise in particular in statistical applications. 

The  empirical law $\S_n$
of the environment's chain $(\bar \w_k)$ , defined as 
\begin{equation}\label{eq:empir}   \S_n 
=  \frac 1n \sum_{k=1}^n \delta_{  \overline{\omega}_k  } \,,
\end{equation}
 allows to represent Birkhoff sum  of  $F: \Omega \to \R$ along the chain  as  an integral
$$
\frac 1n\sum_{k=1}^nF (\bar{\omega}_k) = \int_\Omega F d\S_n \, .
$$ 
Then, $\S_n$ is a  random  element of ${\mathcal P}(\Omega)$ depending both on $\omega$ and $X$.
Our main result, 
Theorem \eqref{thm:main}, states that the following convergence in distribution  in the space $ {\mathcal P}(\Omega)$ equipped with the topology of the weak convergence of probability measures holds:
\begin{equation*}
  \S_n 
 \stackrel{\LL}{\longrightarrow} 
   \S_\8, 
\end{equation*}
where the law of the random measure $\S_\8$ is precisely defined in \eqref{eq:sinfty}.
In particular, for every $F: \Omega \to \R$ continuous and bounded, the following convergence in law holds: 
\begin{equation}\label{inlaw}
\frac 1n\sum_{k=1}^nF (\bar{\omega}_k) \stackrel{\LL}{\longrightarrow} \int_\Omega F d\S_\8 \, .
\end{equation} 
Note that when $\Omega$ is equipped with the Hilbert's cube distance $d$, 
 all the functions $F$ depending only on the finite numbers of coordinates of $\w$ are continuous w.r.to $d.$ 

Despite the fact that \eqref{inlaw} gives only an "in law" version of the ergodic theorem,
 in many examples the method of the environnement viewed from the particle can be used in a very similar to the ballistic case way.
The point is that in situations where the limit in \eqref{inlaw} is deterministic, the  convergence actually  holds in probability. Hence the environment's method of the example above can be performed almost in a same way, replacing the a.s. convergence by the convergence in probability for the second sum.  We give such examples in Section \eqref{sec:ex}. On the other hand, besides the environment's method, often only the integrability properties of the limit \eqref {inlaw} are of interest, so the knowledge of its distribution can be sufficient. 
To define precisely the limit random measure $S_{\8}$ we need to introduce the notion of the potential and the infinite valley.


\subsection{Potential and infinite valley}
Let $\rho_x, x\in\Z$ be given by \eqref{eq:rho} and define the
 potential  $V = \{ V(x) \, : \,  x \in \Z \}$ 
 by 
\begin{equation} \label{equa:Pot} 
V(x) = \left\{
    \begin{array}{ll} \sum_{y=1}^x \log \rho_y & \mbox{if $x>0$,} \\ 0
& \mbox{if $x=0$,} \\ - \sum_{y=x+1}^0 \log \rho_y& \mbox{if $x<0$.}
    \end{array} \right.
\end{equation}
 Then, $V$ is a (double-sided) random walk, 
 an  example of a realisation of $V$ can be seen on Figure~\ref{figu:potential}.
Setting $C(x,x+1)=\exp[-V(x)]$,
for any  integer $x$, under quenched law the  Markov chain $X$  is an
electric network in the sense of \cite{DoSn} or \cite{LePeWi}, where
$C(x,x+1)$ is the conductance  of the (unoriented) bond $(x,x+1)$.  In
particular, the measure $\mu$ defined as 
\begin{equation*} 
\mu(0)=1,\quad \mu(x) = \exp[-V(x-1)]+\exp[-V(x)], \quad x\in\Z^*_+,
\end{equation*}

is a reversible and invariant measure for the Markov chain $X$. 
Define the right border $c_n$ of the ``valley'' with depth $\log n +
  (\log n)^{1/2}$ as the random variable
\begin{equation} \label{eq:c_n}
  c_n = \min \big\{x \geq 0 \,  : \, V(x) - \min_{0\leq y \leq x} V(y)
  \geq \log n + (\log n)^{1/2} \big\}, 
\end{equation}
and the bottom $b_n$ of the ``valley''  as
\[ 
  b_n = \min  \big\{x \geq 0 \, :  \, V(x) = \min_{0 \leq  y \leq c_n}
  V(y)\big\}. 
\]
On Figure~\ref{figu:potential}, one can see a representation of $b_n$ and
$c_n$.
The salient probabilistic feature of a recurrent RWRE is the strong localisation revealed by \cite{sinai}. Considered on the spacial scale $\ln^2 n$ the RWRE becomes localized  near $b_n. $  We are  interested in the shape of  the ``valley'' $(0,b_n,c_n)$
when  $n$ tends  to infinity  and we  recall the  concept  of infinite
valley introduced by \cite{Golosov}.

\begin{figure}
\begin{center}
\includegraphics[trim = 48mm 175mm 25mm 45mm, clip,width=\textwidth]{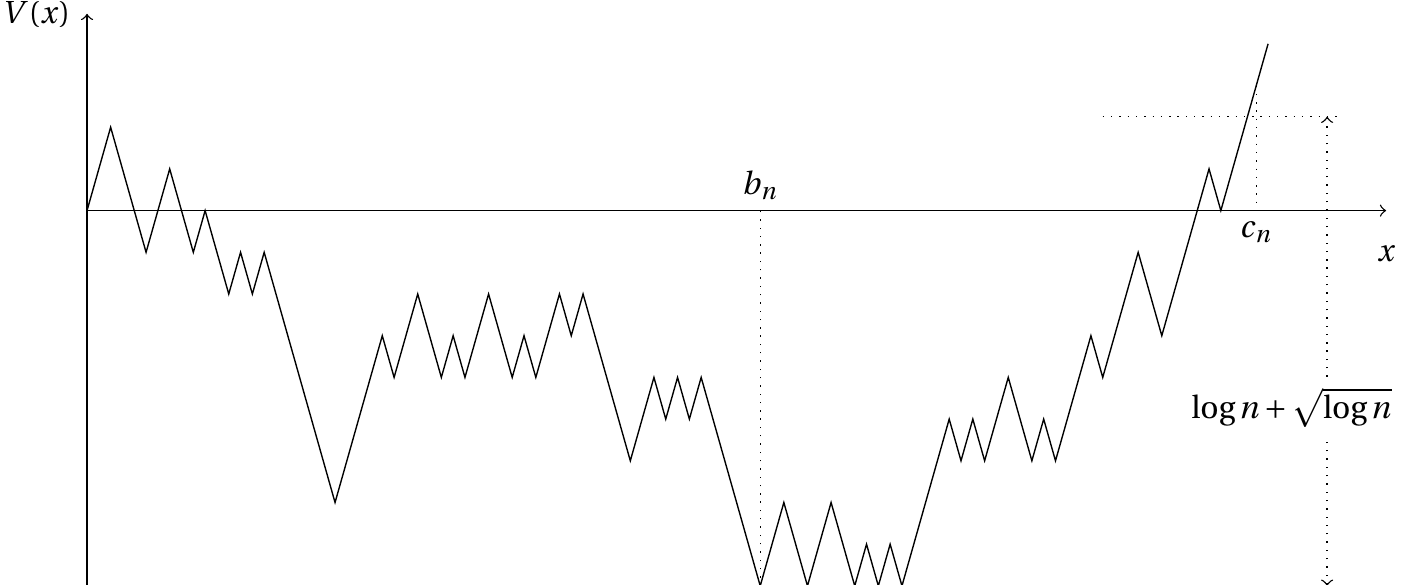}
\end{center}
\caption{Example  of  potential  derived  from  a  Temkin random  environment
  with parameter $a=0.3$. Simulation with $n=1000$.}
\label{figu:potential}
\end{figure}

Let $\widetilde  V =  \{ \widetilde V(x)  \, :  \, x \in  \Z \}$  be a
collection of random variables  distributed as $V$ conditioned to stay
positive   for   any   negative   $x$,   and   non-negative   for   any
non negative~$x$.  Such events having probability zero, a formal definition is using Doob's $h$-transform (see \cite{Golosov}[  Lemma 4], \cite{Bertoin}). 
It has been shown in \cite {Golosov},
that the finite dimensional distributions of $\{  [V(b_n+x)-V(b_n)]\1_{\{-b_n,\ldots, c_n-b_n-1\}}(x)\ ;\ x\in\Z      \}$  converges to those of $\{ \widetilde V(x)  \, :  \, x \in  \Z \},$ moreover, (\cite {Golosov}, pp. 494-495)
\begin{equation}\label{finitV}
\sum_{x\in\Z}\exp\left (-\widetilde V(x)\right )<\infty.
\end{equation}
Besides for the fidi convergence above, it is not true in general that the sequence of the  infinite vectors $\{ [ V(b_n+x)-V(b_n)]\1_{\{-b_n,\ldots, c_n-b_n-1\}}(x)\ ;\ x\in\Z   \}$ converges in law to $\{\tilde V(x),\, x\in\Z\}$.  But if we consider instead the sequence of elements of $(\ell^1, \|\cdot\|_1)$ given by
  \[
\Xi_n:= \left \{\exp[-(V(b_n+x)-V(b_n))]\1_{\{-b_n,\ldots, c_n-b_n-1\}}(x)\ ;\ x\in\Z \right. \}
 \]
 we can show ( Proposition \ref{prop:tight}) that
the sequence of laws $P_{\Xi_n}$ of $\Xi_n$ is tight, and hence  $\Xi_n$ converges in distribution to  
$\ \{\exp[-\widetilde V(x)]\ ;\ x\in\Z\}.$ This  is done in Theorem \eqref{th: weakconv}, which is 
 a key auxiliary result for the proof of Theorem \eqref{thm:main}. In the next subsection we formulate this theorem precisely. 

 \subsection{Assumptions and main result}
\begin{assu}\label{as:rec}
$\E\log \rho_0=0
$
\end{assu}
We already mentioned that under Assumption \eqref{as:rec} for $\P$-almost $\w$ the Markov chain $X$ is recurrent. We also need to assume
\begin{assu}\label{as:var}
\begin{description}
\item {(i)}
$\P(\delta_0\leq \w(0)\leq 1-\delta_0)=1 \quad \mbox{for some}\quad \delta_0\in(0,1),$
\item {(ii)}
$ \Var (\log \rho_0)>0.
$
\end{description}
\end{assu}

  The condition $(i)$ is technical and commonly admitted, whereas $(ii)$ excludes the deterministic case. 
Moreover, in the proof of Proposition \eqref {prop:tight}, Theorem\eqref{th: weakconv} and hence in Theorem \eqref{thm:main} we need to assume the following technical assumption:
\begin{assu}\label{as:aritm}
The distribution of $\log \rho_0$ is arithmetic, i.e. concentrated on $\{nh;\;  n\in\Z\},$ with some $h>0.$  
\end{assu}

%
Let $\widetilde \omega=\{\widetilde\omega(x),\ x\in\Z\}$ be the {\it environment of the walk in the infinite valley}:   
\begin{equation} \label{equa:tildeOmega} 
  \widetilde \omega (x) 
   = \frac {\exp[-\widetilde V(x)] }{
    \exp[-\widetilde V(x)] +\exp[- \widetilde V(x-1)]},\quad x \in \Z.
\end{equation}

Let $\tilde \nu$  be a  probability measure on $\Z$  defined by
\begin{equation} \label{equa:Nu} 
 \tilde\nu(x) = \frac{ \exp[-\widetilde 
   V(x-1)]+\exp[-\widetilde  V(x)] }{2\sum_{z\in  \Z} \exp[-\widetilde
   V(z)]}, \quad x \in \Z . 
\end{equation}

Thanks to \eqref{finitV} the probability measure \eqref {equa:Nu} is well defined, and is a 
  stationary (and reversible)  distribution of
a random walk in in the "infinite valley", i.e. the walk governed by the environnement $\tilde\w$.

Define for $n\in\N$ and $x\in\Z,$ the local time of the walk in the position $x$:

\begin{equation}\label{eq:nun} 
\xi(n,x)=
\sum_{k=1}^{n}\1\{X_k=x\}\;,
\end{equation}

Note that the  empirical law \eqref {eq:empir} of the environment seen from the walker can be expressed using the local time as
\begin{equation}\label{eq:empir1}   \S_n 
=
\sum_{x \in Z} \frac{\xi(n,x)}{n} \delta_{ T_x  {\omega}}\  
\,.
  \end{equation}
Denote
\begin{equation}\label{eq:sinfty}
\S_\8:=\sum_{x \in Z} \tilde\nu(x) \delta_{ T_x  \widetilde{\omega}}\,.
\end{equation}
Let $\Omega:=[0,1]^{\Z}$ be provided with the distance $d(\w,\w')=\sum_{x\in\Z}2^{-|x|}|\w(x)-\w'(x)|.$
\begin{theo}\label{thm:main}  Under Assumptions \ref{as:rec}, \ref{as:var}  and \ref{as:aritm}, the empirical law of the environment seen from the walker converges in distribution, as $n \to \8$:  
 \begin{equation}\label{eq:p1}
  \S_n 
 \stackrel{\LL}{\longrightarrow} 
   \S_\8
\end{equation}
  in the space $ {\mathcal P}(\Omega)$ equipped with the topology of the weak convergence of probability measures.
  \end{theo}
  Note that in particular, for every $F:\Omega\to\R$ continuous and bounded, the "weak" ergodic theorem \eqref{inlaw} holds, and therefore 
 for every $m\in\N,$
  $g:[0,1]^{2m+1}\to\R,$ continuous,
\[
\frac 1n\sum_{k=1}^ng(\omega(X_{k}-m),\ldots,  \w(X_{k}+m))
\stackrel{\LL}{\longrightarrow} \sum_{x \in \Z}g(\widetilde{\omega}(x-m),\ldots, \widetilde{\omega}(x+m))\tilde \nu(x)\;.
\]
Again in particular, for every  $f:[0,1]\to\R,$ continuous, 
\[
\frac 1n\sum_{k=1}^nf(\omega(X_{k}))
\stackrel{\LL}{\longrightarrow} \sum_{x \in \Z}f(\widetilde{\omega}(x))\tilde \nu(x)\;.
\]


 Denote by ${\mathcal E}$ the expectation with respect to the law of $\tilde V=(\tilde V(x))_{x\in\Z}$ and let us define $\IS \in {\mathcal P}(\Omega)$ by
$$
\int_\Omega F d\IS = {\mathcal E}\left[ \int_{\Omega} F d\S_\8 \right]= \sum_{x \in \Z} {\mathcal E}\big[\tilde \nu(x) F(T_x \widetilde{\omega}) \big] \;,
$$
for bounded  $F:\Omega\to\R$. We can view $\IS$ as the $\mathcal E$-expectation of $\S_\8$. 


\begin{prop} \label{thm:inv}
The probability $ \IS$ is invariant and reversible for the  Markov chain $( \overline{\omega}_k, k \geq 0)$ in $\Omega$.  The measures $\P$ and $\IS$ are mutually singular. 
\end{prop}
 The invariant probability, which is a  limit law in the ballistic case, 
 is absolutely continuous with respect to the law of the environment $\P,$  see \cite[P. 273]{Molchanov-StFlour}. 
The one we find here is the first one to be obtained as a limit in the case of zero velocity, and it is singular with respect to $\P.$ 

The proof of Theorem \eqref{thm:main} is partially inspired by the paper \cite{GPS} concerning the convergence of centred  local times: $(\frac{\xi(n,b_n+x)}{n},\; x\in\Z)$ to $(\tilde\nu(x)),\; x\in\Z$, but the main ingredient, Proposition \eqref{prop:main} giving the tightness of $\left \{\exp[-(V(b_n+x)-V(b_n))]\1_{\{-b_n,\ldots, c_n-b_n-1\}}(x)\ ;\ x\in\Z \right \}$ is new. In its turn, one part of the proof of Proposition \eqref{prop:main} is inspired by the paper of \cite{riter}on the growth of random walks conditioned to stay positive.
\subsection{Structure of the paper}
In Section\eqref{sec:ex} we show how the environment's method can be deduced from Theorem\eqref{thm:main}. Namely we  proove the LLN (Proposition \eqref{prop:lln}) and the mixed CLT(Proposition \eqref{prop:clt}) for sums  $ \sum_{k=1}^nf(\Delta X_k). $   Section \eqref{sec:profthm1} is focused on the proof of Theorem \eqref{thm:main}. Proposition \ref{thm:inv} is proven in Section\eqref{sec:inv}. Auxiliary results for the proof of Theorem \ref{thm:main}, and in particular Proposition \eqref{prop:main} are proven in Section \eqref{sec:aux}. 
\section {Examples: environnement's method}\label{sec:ex}

\subsection{Law of large numbers for functions of the steps}\label{subsec:LLN}
\begin{prop}\label{prop:lln}
Let $f:\{-1;1\}\rightarrow \R.$ Denote $\Delta X_k=X_{k+1}-X_k,$  $k\in\N.$ Then the following convergence in annealed probability  holds :
\begin{equation*}
\frac 1n\sum_{k=1}^nf(\Delta X_k)\stackrel{\PP}{\longrightarrow}\frac{f(1)+f(-1)}2,\quad \quad {n\to\infty}.
\end{equation*}
In particular, $X_n/n\stackrel{\PP}{\longrightarrow}0,\quad{n\to\infty}.$
\end{prop}
\begin{proof}
Denote \begin{equation}\label{eq:Fk}
\FF_n=\sigma\{\Delta X_0,\ldots, \Delta X_n,\;   \w({X_0}),\ldots, \w({X_{n+1}})\}
\end{equation}
and let's write the martingale difference decomposition :
\begin{equation}\label{eq:decompos}
\frac 1n\sum_{k=1}^nf(\Delta X_k)=\frac 1n\sum_{k=1}^n\left(f(\Delta X_k)-\EE(f(\Delta X_k)|\FF_{k-1})\right)+\frac 1n\sum_{k=1}^n\EE(f(\Delta X_k)|\FF_{k-1}).
\end{equation}
Then $D_k:=f(\Delta X_k)-\EE(f(\Delta X_k)|\FF_{k-1});\, k\in\N$ are centred, uniformly bounded and  non-correlated, (the last can be immediately seen for $D_k$ and $ D_m$, $k<m$  by conditioning on $\FF_{m-1}).$  Hence 
\begin{equation}\label{eq:l2}
 \frac 1n\sum_{k=1}^n\left(f(\Delta X_k)-\EE(f(\Delta X_k)|\FF_{k-1})\right)\stackrel{\L^2}{\longrightarrow} 0. 
 \end{equation}
%
%
Remark that 
$$\EE(f(\Delta X_k)|\FF_{k-1}))=f(1)\w({X_k})+f(-1)(1-\w({X_k})).$$
Theorem \ref{thm:main} gives the following convergence in distribution :
\begin{equation}\label{eq:indist}
\frac 1n\sum_{k=1}^n\EE(f(\Delta X_k)|\FF_{k-1})\stackrel{\LL}{\longrightarrow} \sum_{x\in\Z}\left(f(1)\tilde \w(x)+
f(-1)(1-\tilde \w(x))\right )\tilde\nu(x)=\frac{f(1)+f(-1)}2
.
\end{equation}
 Indeed, using the definitions \eqref{equa:Nu} and \eqref{equa:tildeOmega} ,
$$\sum_{x\in\Z}\tilde \w(x)\tilde\nu(x)=\sum_{x\in\Z}(1-\tilde \w(x))\tilde\nu(x). $$  Using \eqref{eq:l2} and \eqref{eq:indist} together in \eqref{eq:decompos} this completes the proof.

\end{proof}

\begin{prop}\label{prop:clt}
Let $f:\{-1;1\}\rightarrow \R$ and $(\FF_n)$ defined by \eqref{eq:Fk}.
Then the following mixing CLT holds:
\begin{equation}\label{eq:clt}
\frac 1{\sqrt n}\sum_{k=1}^n\left(f(\Delta X_k)-\EE(f(\Delta X_k)|\FF_{k-1})\right)\stackrel{\LL}{\longrightarrow} Z,
\end{equation}
where $Z$ is a random variable with the characteristic function $\phi_Z(t)=\EE(\exp(-\frac 12 \eta^2t^2),$ and
$\eta$ is a random variable defined by: $\eta^2 \stackrel{\LL}{=}(f(1)-f(-1))^2\sum_{x\in\Z}\tilde \w(x)(1-\tilde \w(x))\tilde\nu(x)$.
That is $Z\stackrel{\LL}{=}\eta U$ where $\eta$ and $U$ are independent and $U\sim{\cal{N}}(0,1).$
\end{prop}
\begin{proof}
In this proof we will rely on Theorem 3.4 from \cite{Hall_Heyde}.  Define for $k=1,\ldots, n$
$$D_{nk}:=\frac 1{\sqrt n}\left (f(\Delta X_k)-\EE(f(\Delta X_k)|\FF_{k-1})\right )=\frac 1{\sqrt n}\left (f(\Delta X_k)-f(1)\w({X_k})-f(-1)(1-\w({X_k)})\right ).$$
Let $S_{n0}=0,$  $S_{ni}=\sum_{k=1}^{i}D_{nk}$ and  put for $i=1,\ldots, n,$ $\FF_{ni}:=\FF_i.$ Then $S_{ni} $ is adapted to $\FF_{i}.$ Let $U_{ni}^2=\sum_{k=1}^{i}D_{nk}^2.$ 
Denote
$$\GG_n=\sigma\{  \w({X_0}),\ldots, \w({X_{n+1}})  \},$$
Clearly $\GG_n\subset\FF_n.$  Let for $i=1,\ldots, n,$
$$\GG_{n,i}:=\FF_{ni}\vee\GG_n=\FF_i\vee\GG_n=\sigma(\Delta X_1,\ldots,\Delta X_i,\w({X_0}),\ldots, \w({X_{n+1}})\}.$$
Next we have
\begin {equation}\label{eq:estimate}
 \max_{i=1\ldots n}|D_{ni}|\leq \frac {2\|f\|}{\sqrt n}\longrightarrow 0\quad\mbox{ and}\quad  \EE\left (\max_{i=1\ldots n} |D_{ni}|^2\right ) \leq \frac {2\|f\|^2}{ n}.
 \end{equation}{eq:estimate}
Define a random sequence $(u_n^2)$ by
$$u_n^2=\sum_{k=1}^n\EE\left( D_{nk}^2|\FF_{k-1}\right).$$
It is easy to see that
 $$
\EE\left( D_{nk}^2|\FF_{k-1}\right)=\frac 1n(f(1)-f(-1))^2\w({X_k})(1-\w({X_k})),$$
hence the sequence  $(u_n^2)$ is $(\GG_n)-$ adapted.
In order to show the convergence in probability ( the condition of Theorem 3.4 $(3.28)$ from \cite{Hall_Heyde}):
\begin{equation}\label{eq:eqproba}
U_{nn}^2-u_n^2=\sum_{k=1}^n\left(D_{nk}^2-\EE(D_{nk}^2|\FF_{k-1})\right )\stackrel{\P}{\longrightarrow} 0,
\end{equation}
we remark that $\left (D_{nk}^2-\EE(D_{nk}^2|\FF_{k-1})\right );\ k=1,\ldots, n,$ are centred, $(\FF_k)-$ adapted  and non correlated. Indeed, if $k<m,$ then $k\leq m-1$ and
\begin{align*}
&\EE\left [ \left (D_{nk}^2-\EE(D_{nk}^2|\FF_{k-1})\right )\left (D_{nm}^2-\EE(D_{nm}^2|\FF_{m-1})\right )    \right ]=\\
&\EE\left [ \left (D_{nk}^2-\EE(D_{nk}^2|\FF_{k-1})\right )\EE\left [\left (D_{nm}^2-\EE(D_{nm}^2|\FF_{m-1})\right ) |\FF_{m-1} \right ]  \right ]=0.
\end{align*} 
Hence, using \eqref{eq:estimate}, $U_{nn}^2-u_n^2$ converges to $0$ in $\L^2$, and hence in probability:
\begin{align*}
\EE \left( \sum_{k=1}^n\left(D_{nk}^2-\EE(D_{nk}^2|\FF_{k-1})\right )  \right )^2=\sum_{k=1}^n\EE\left(D_{nk}^2-\EE(D_{nk}^2|\FF_{k-1})\right )^2\leq \frac {2n\|f\|^4}{ n^2}\to 0.
\end{align*}
Then for all $i=1\ldots n,$ $\E\left ( D_{ni}|\GG_{n,i-1}\right)=0,$  hence the condition $(3.29)$ of Theorem 3.4 from \cite{Hall_Heyde} follows.


  Applying Theorem \eqref{thm:main} we see that 
  $$u_n^2=(f(1)-f(-1))^2\frac 1n\sum_{k=1}^n\w({X_k})(1-\w({X_k}))\stackrel{\LL}{\longrightarrow} (f(1)-f(-1))^2\sum_{x\in\Z}\tilde\w(x)(1-\tilde\w(x))\tilde\nu(x).$$Then, using \eqref{eq:eqproba},
  $$U^2_{nn}=(U^2_{nn}-u_n^2)+u_n^2\stackrel{\LL}{\longrightarrow}\eta^2:= (f(1)-f(-1))^2\sum_{x\in\Z}\tilde\w(x)(1-\tilde\w(x))\tilde\nu(x).$$
and the theorem follows.
\end{proof}

\section{Proof of Theorem \ref{thm:main}}\label {sec:profthm1}

\begin{proof} 
  By definition, claim \eqref{eq:p1} is equivalent to 
\begin{equation}\label{cvSh}
\lim_{n\to \8} {\mathbf E} G( \S_n)=  {\mathcal E} G(\S_\8)
\end{equation}
 for all bounded continuous $G: {\mathcal P}  (\Omega) \to \R$. We first observe that   it is sufficient to prove \eqref{cvSh}  for all $G$ of the form 
\begin{equation}\label{Gsimple}
G( {\mathcal S} ) = \sum_{l=1}^n\left( \int_\Omega F_1 d {\mathcal S}\times\ldots\times  \int_\Omega F_l d {\mathcal S}\right)
\end{equation}
with arbitrary integers $n$, $m,$ $l$  and $F_k: [0,1]^{2m+1} \to \R$ 
continuous ($1\leq k \leq l$).
 Indeed, let $d$ be a distance on $\Omega$ defined by
$$d(\w,\w')=\sum_{x\in\Z}\frac 1{2^x}|\w(x)-\w'(x)|.$$ 
Then $(\Omega, d)$ is a compact separable metric space and hence $({\mathcal P}  (\Omega),\rho)$, endowed with the Prohorov metric $\rho$, is a compact separable metric space too.
The set ${\mathcal G}$ of functions $G$ of the form \eqref{Gsimple}  is an algebra of continuous functions on the compact metric space   ${\mathcal P}  (\Omega)$  which contains constant functions and separates the points. 
By Stone-Weierstrass Theorem, this set is dense in the space  ${\mathcal C} ( {\mathcal P}  (\Omega); \R )$ for the supremum norm, and then it suffices to prove \eqref{cvSh} for such $G$'s.
This, in turn, is equivalent to prove the convergence in distribution:
\begin{equation}\label{Fsimple}
\left( \int_\Omega F_1 d {\mathcal S}_n,\ldots,  \int_\Omega F_l d {\mathcal S}_n\right) \stackrel{\LL}{\longrightarrow} 
\left( \int_\Omega F_1 d {\mathcal S}_\8,\ldots,  \int_\Omega F_l d {\mathcal S}_\8 \right) 
\end{equation}
as $n \to \8$.
Indeed, using Cramer-Wold device \eqref {Fsimple}  is equivalent to 
\begin{equation*}
\forall (t_1,\ldots t_l)\in\R^l,\quad 
\sum_{i=1}^l t_i \int_\Omega F_i d {\mathcal S}_n \stackrel{\LL}{\longrightarrow} 
\sum_{i=1}^l t_i \int_\Omega F_i d {\mathcal S}_\8,
\end{equation*}
 and finally, as $\sum_{i=1}^l t_i F_i$ is a continuous function on $\Omega$, depending only on the finite number of coordinates, we only need to prove that 
\begin{equation}\label {FFsimple}
\forall m\in\N^*,\;\forall F\in\C_b([0,1]^{2m+1}),\;  \quad\quad
 \int_\Omega F d\S_n\stackrel{\rm law}{\longrightarrow} \int_\Omega F d {\mathcal S}_\8.
\end{equation}
Bellow we we give the proof of \eqref{FFsimple}, wich is separated on $3$ main steps.\\
{\bf Step 1:{\it Approximation in probability of $\int_\Omega F d\S_n$.}}\\ 
 For $F$ as in \eqref{FFsimple}, using \eqref {eq:nun}  and \eqref{eq:empir1} let's write $\S_n(F)$ in the "spatial" form 
 \begin{equation}\label{eq:SnFspat}
\S_n(F)=\sum_{x\in\Z}F(T_x \omega) \frac{\xi(n,x)}{n}=\sum_{x\in\Z}F(\w({-m+x}),\ldots,\w({m+x})) \frac{\xi(n,x)}{n}.
 \end{equation}
Fix $n\in \N^*$ and denote $\mu_n=\mu_n^\w$ the random probability measure  on $\Z_{+}$ 
\begin{equation}\label{eq:mu_n}
\mu_n(x):=\left\{
\begin{array}{ll}\frac1{Z_n}\left(e^{-V(x)}+e^{-V(x-1)}\right)&\mbox{if $0<x<c_n$,}\\
\frac1{Z_n}&\mbox{if $x=0$,}\\
\frac1{Z_n}e^{-V(c_n-1)}&\mbox{if $x=c_n$,}\\
0&\mbox{if $x\notin \{0,\ldots,c_n\},$}
\end{array}\right.
\end{equation}
where $Z_n=2\sum_{x=0}^{c_n-1}e^{-V(x)},$ where $c_n$ and $V$ are respectively defined  by \eqref{eq:c_n} and \eqref{equa:Pot}.
The  point is that the local times $\frac{\xi(n,x)}{n} , \, x \in \Z_+ $
 can be approached in probability by the quantities $\mu_n(x),\ x\in\Z_+$ . This argument was found by \cite {GPS}.  Here we show that more generally, the additive functional $\S_n(F)$
 can be approached in probability by $\Sigma_n(F)=\int_{\Omega}Fd\Sigma_n$ with
 \begin{equation} \label{def:Sigma_n}
 \Sigma_n=
\sum_{x \in Z} \mu_n(x)\delta_{ T_x  {\omega}}\;.
\end{equation}
  Namely, Proposition \ref{prop:deviation} states that $\forall \eps >0,$
 \begin{equation}\label{eq:cvproba}
 \PP\left(\left|\S_n(F)- \Sigma_n(F)
\right|>\eps  \right)\to 0 .
 \end{equation}
 Note that $\S_n$ depends on the walk and on the environment, whereas $\Sigma_n$ depends only on the environment. 
The next three steps allow to show the convergence in law:  
 $
\Sigma_n(F)\stackrel{\LL}{\longrightarrow} \S_\8 (F).$ \\
{\bf Step 2:} {\it Expressing 
   $\Sigma_n(F)$
   as a continuous function of a weakly convergent sequence.}\\
   Note that
   \[
  \Sigma_n(F)=\sum_{x\in\Z} F(T_x \omega)\mu_n(x)=\sum_{x\in\Z} F(T_{b_n+x} \omega)\mu_n(b_n+x).\]
Denote by $\Xi_n$ the random element in $\ell^1$ given by:
  \[
 \Xi_n:=\left \{\exp[-(V(b_n+x)-V(b_n))]\1_{\{-b_n,\ldots, c_n-b_n-1\}}(x)\ ;\ x\in\Z \right. \}.
 \]
Both  $\mu_n(b_n+x)$ and $\omega(b_n+x):\ x=-b_n,\ldots, c_n-b_n-1\ $ 
can  be expressed in terms of $\Xi_n$:  
\begin{align}
\mu_n(b_n+x)
=&\frac {\Xi_n(x)+\Xi_n(x-1)}{2\sum_{y\in\Z} \Xi_n(y)}
,\end{align}
and
\begin{align} \label{equa:tildeOmegab} 
   \omega (b_n+x) 
    =&\frac{\Xi_n(x)}{\Xi_n(x)+\Xi_n(x-1)}.
\end{align}
Thus, $\Sigma_n(F)= H_F( \Xi_n)$
where $H_F:{\ell^1}\to\R$ is continuous.
%
 In Theorem \ref{th: weakconv} we show that the distribution of $\Xi_n$ converges weakly to that of $\{\exp[-\widetilde V(x)]\ ;\ x\in\Z\}$ in this space. 
 Together with the continuity on $\ell_1$ of $\Sigma_n(F)= H_F( \Xi_n)$  that gives the convergence in law 
 \begin{equation}\label{eq:cvH}
\Sigma_n(F)\stackrel{\LL}{\longrightarrow}\S_\8(F).
\end{equation}
{\bf Step 3} {\it Conclusion:} Using \eqref{eq:cvproba} and \eqref{eq:cvH} we conclude that 
$  \S_n(F) \stackrel{\LL}{\longrightarrow}
 \S_\8(F).
$
This ends the proof of Theorem \ref{thm:main}. \end{proof}

\section{Auxiliary results for the Proof of Theorem \ref{thm:main}} \label{sec:aux}
\subsection{Approximation in probability}\label{sec:dev}
Let $\mu_n$ be given by \eqref {eq:mu_n}. Note that $\mu_n$ is a probability measure and that it is invariant for the chain $\tilde \X^n=(\tilde X_t^n)_{,t\in\N}$   with value in $\{0,\ldots, c_n\}$ and with transition density given by
$\tilde p^{\w,n}:\{0,\ldots, c_n\}^2\to [0,1]$  given  by:
 $$\tilde p^{\w,n}(0,1)=\tilde p^{\w,n}(c_n,c_{n-1})=1$$ and if $\ x\in \{1,\ldots,c_n-1\};$
\[\tilde p^{\w,n}(x,x+1)=\w_x\;, \quad \ \tilde p^{\w,n}(x,x-1)=1-\w_x\;.\]
For $x\in \{0,\ldots, c_n\}$ we denote  by $\tilde P^{\w,n}_x$  the law on $\{0,\ldots, c_n\}^{\N}$ of the Markov chain $\tilde\X^n$ starting from $x$.

\begin{prop}\label{prop:deviation}
Let $\S_n$ given by \eqref{eq:empir1} and $\Sigma_n$ by\eqref{def:Sigma_n}.
For all $F:[0,1]^{2m+1}\to\R$ continuous, all $\eps >0 $ 
we have
\[\PP\left(\left|  \S_n(F) - \Sigma_n(F)\right|>\eps  \right)\to 0 .\]
\end{prop}
\begin{proof}
Let $n\in\N^*$ and $\w\in\Omega$ be fixed. Denote $T^0=0.$ For $y\in\Z^+,$ denote 
$$T_{y}=T^1_{y}:=\inf\{t>0,\  X_t=y\}$$ and 
$$\forall k>1,\quad T^k_{y}:=\inf\{t>T^{k-1}_{y},\  X_t=y\}$$
the times of successive visits of  $y$ by the walk. Using the recurrence of $\X$, 
 $$\forall k\in\N^*,\ \forall y\in\Z^+,\quad  T^k_{y}<\infty \quad\PP-a.s.$$ 
Denote $k_n$ the number of visits of $b_n$ by the walk before the time $n:$
$$k_n=\sum_{t=0}^n\1_{\{b_n\}}(X_t)$$ and let $\mu_n$ be given by \eqref {eq:mu_n}. The  random walk with value in $\{0,\ldots, c_n\}$, reflected in $0$ and $c_n,$ admits  $\mu_n$ as an invariant measure, and $k_n/n$ can be compared with $\mu_n$.   
First of all we obtain a bound on the quenched  probability of the deviation of $k_n/n$ from $\mu_n$. 
\begin{align}\label{al:kndecomp}
&P^{\w}_0\left(\left|\frac {k_n}n-\mu_n(b_n)\right|>\eps\right)\leq\\ \nonumber
&P^{\w}_0\left(\left|\frac {k_n}n-\mu_n(b_n)\right|>\eps,\ T_{b_n}<n\eps/2,\ T_{c_n}>n\right)+P^{\w}_0\left(T_{b_n}\geq n\eps/2 \right)+P^{\w}_0\left(T_{c_n}\leq n\right).
\end{align}
For the first term of the inequality \eqref{al:kndecomp} we can write:
\begin{equation}\label{al:kndecomp1}
P^{\w}_0\left(\left|\frac {k_n}n-\mu_n(b_n)\right|>\eps,\ T_{b_n}<n\eps/2,\ T_{c_n}>n\right)\leq P^{\w}_0(B_1)+P^{\w}_0(B_2);
\end{equation}
Where we have denoted
\begin{align*}
&B_1:=\left\{ {k_n}\geq [n(\mu_n(b_n)+\eps)]+1,\ T_{b_n}<n\eps/2,\ T_{c_n}>n\right\};\\ \nonumber
&B_2:=\left\{ {k_n}\leq [n(\mu_n(b_n)-\eps)],\ T_{b_n}<n\eps/2,\ T_{c_n}>n\right\}.
\end{align*}
Both events $B_1$ and $B_2$ concern with the part of the trajectory $X_0,\ldots, X_n,$ where the value $c_n$ did not occurred.
Hence $P^{\w}_0(B_i)=\tilde P^{\w,n}_0(B_i),\ i=1,2.$ Then, using the definition of $(T^k_{b_n}),$ strong Markov property  and Markov inequality we can write for $n \geq 2/\eps$
\begin{multline}\label{al:B1}
P^{\w}_0(B_1)=\tilde P^{\w,n}_0(B_1)\leq \tilde P^{\w}_0\left ( T^{[n(\mu_n(b_n)+\eps)]+1}_{b_n} \leq n\right )\leq\\
\tilde P^{\w,n}_0\left ( \sum_{k=1}^{[n(\mu_n(b_n)+\eps)]}(T^{k+1}_{b_n}-T^{k}_{b_n})\leq n\right )\leq 
\tilde P^{\w,n}_{b_n}\left ( \sum_{k=1}^{[n(\mu_n(b_n)+\eps)]}\eta_k \leq -\frac{n\eps - 1}{\mu_n(b_n)}\right )\\ \leq\frac{n(\mu_n(b_n)+\eps)\mu_n^2(b_n)\widetilde \Var^{\w,n}_{b_n}\eta_1}{(n\eps - 1)^2}\leq 4\frac{(1+\eps)}{\eps^2}\frac{\widetilde\Var^{\w,n}_{b_n}\eta_1}n,
\end{multline}
where 
\[\eta_k=T_{b_n}^{k+1}-T_{b_n}^{k}-1/\mu_n(b_n)\]
are i.i.d. and centered under $\tilde P^{\w,n}_{b_n},$ since $\mu_n$ is the invariant probability for $\tilde p^{\w,n}$ and ${\tilde E^{\w,n}_{b_n}T_{b_n}^1=\frac 1{\mu_n(b_n)}.}$
Similar arguments give
\begin{multline}\label{al:B2}
P^{\w}_0(B_2)=\tilde P^{\w,n}_0(B_2)\leq \tilde P^{\w}_0\left ( T^{[n(\mu_n(b_n)-\eps)]}_{b_n}\geq n,\ T_{b_n}<n\eps/2\right )\leq\\
\tilde P^{\w,n}_0\left ( \sum_{k=1}^{[n(\mu_n(b_n)-\eps)]-1}(T^{k+1}_{b_n}-T^{k}_{b_n})\geq n(1-\eps/2)\right )\leq 
\tilde P^{\w,n}_{b_n}\left ( \sum_{k=1}^{[n(\mu_n(b_n)-\eps)]-1}\eta_k\geq \frac{n\eps}{2}\right )\\
\leq 4\frac{n(\mu_n(b_n)-\eps)\widetilde \Var^{\w,n}_{b_n}\eta_1}{n^2\eps^2}\leq 4\frac{(1-\eps)}{\eps^2}\frac{\widetilde\Var^{\w,n}_{b_n}\eta_1}n.
\end{multline}
Finally, putting together \eqref{al:kndecomp}, \eqref{al:kndecomp1}, \eqref{al:B1} and \eqref{al:B2} we get:
\begin{align}\label{al:knbound}
&P^{\w}_0\left(\left|\frac {k_n}n-\mu_n(b_n)\right|>\eps\right)\leq
&\frac{8}{\eps^2 n}{\widetilde\Var^{\w,n}_{b_n}T_{b_n}}+P^{\w}_0\left(T_{b_n}\geq n\eps/2\right)+P^{\w}_0\left(T_{c_n}\leq n\right),
\end{align}
Under the quenched law $\w$ is fixed. For fixed $\w\in\Omega$ and $x\in\Z_+$ denote 
\[F_{\w}(x):=F(T_{x}\w),\quad\overline F_{\w}(x):=F_{\w}(x)- \Sigma_n(F)\quad\mbox{and}\quad 
\tilde\eps=\eps/{3\|F\|_{\infty}}. 
\]
We will first obtain a non-asymptotic bound on the quenched probability of the deviation $\left|  \S_n(F) - \Sigma_n(F)\right|
=\left|\frac 1n\sum_{k=0}^n  F(T_{X_k}\w) - \Sigma_n(F)
\right|=\left|\frac 1n\sum_{k=0}^n  F_{\w}(X_k) - \Sigma_n(F)
\right|$:
\begin{align}\label{al:devdecomp1} \nonumber
&P^{\w}_0\left(\left|\frac 1n\sum_{k=0}^n  F_{\w}(X_k) - \Sigma_n(F)
\right|>\eps  \right)\leq P^{\w}_0\left(\left|\frac {k_n}n-\mu_n(b_n)\right|>\eps\right)+P^{\w}_0\left(T_{c_n}<n\right )+
\\ \nonumber
&P^{\w}_0\left(\sum_{k=0}^{T^1_{b_n}-1}| \overline F_{\w}(X_k)| >n\eps/3\ \right)+P^{\w}_0\left(\left|\sum_{k=T_{b_n}}^{T_{b_n}^{k_n}-1}
\overline F_{\w}(X_k)\right| >n\eps/3, \ \left|\frac {k_n}n-\mu_n(b_n)\right|\leq\eps,\ T_{c_n}>n\right )+\\ 
&P^{\w}_0\left ( \sum_{k=T_{b_n}^{k_n}}^{n}|\overline F_{\w}(X_k)|>n\eps/3,\ T_{c_n}>n\right)
\end{align}
Using the definition of $\tilde \eps$ we  see that 
\begin{equation}\label{eq:firsterm}
P^{\w}_0\left(\sum_{k=0}^{T^1_{b_n}-1}|\overline F_{\w}(X_k)| >n\eps/3\right)\leq P^{\w}_0\left(T_{b_n}>\tilde \eps n \right) .
\end{equation}
The law of ${T_{b_n}^{k_n+1}}-{T_{b_n}^{k_n}}$ conditionaly on $\FF_{T_{b_n}^{k_n}}$ is that of $T_{b_n}.$ Also, using the definition \eqref{eq:mu_n} we can see that $\mu_n(b_n)\geq 1/2c_n.$ Hence,
\begin{align}\label{al:lasterm}
\nonumber
&P^{\w}_0\left ( \sum_{k=T_{b_n}^{k_n}}^{n}|\overline F_{\w}(X_k)|>n\eps/3,\ T_{c_n}>n\right) =
\tilde P^{\w}_0\left ( \sum_{k=T_{b_n}^{k_n}}^{n}|\overline F_{\w}(X_k)|>n\eps/3,\ T_{c_n}>n\right)\leq
\\ 
&\tilde P^{\w}_0\left ( \sum_{k=T_{b_n}^{k_n}}^{T_{b_n}^{k_n+1}}|\overline F_{\w}(X_k)|>n\eps/3\right)\leq
\tilde P^{\w,n}_{b_n}\left(T_{b_n}>n\tilde \eps \right)\leq \frac{1}{n\mu_n(b_n)\tilde\eps}\leq
\frac{2 c_n}{n\tilde\eps}.
\end{align}
Now we obtain a bound for the main term of the decomposition \eqref{al:devdecomp1}. 
Denote for $k\in\N^*,$
$$\xi_k:=\sum_{l=T_{b_n}^k}^{T_{b_n}^{k+1}-1}\overline F_{\w}(X_l).$$
Under $\tilde P^{\w,n}_{0}$ the random variables $\ \xi_k, k\in\N^*$ are i.i.d. Their law is that of $\sum_{l=0}^{T_{b_n}-1}\overline F_{\w}(X_l)$ under $\tilde P^{\w,n}_{b_n}$ and they are centered, because 
$$E^{\w,n}_{b_n}\sum_{l=0}^{T_{b_n}-1}F_{\w}(X_l)=\mu_n(F_{\w}(\cdot))E^{\w,n}_{b_n}T_{b_n}=\Sigma_n(F) E^{\w,n}_{b_n}T_{b_n}.$$
Hence $M_m:=\sum_{k=1}^{m} \xi_k; \ m\in\N^*$ is a square-integrable martingale under $\tilde P^{\w,n}_{0}$. Using Kolmogorov inequality we get:

\begin{align}\label{al:main}
&P^{\w}_0\left(\left|\sum_{k=T_{b_n}}^{T_{b_n}^{k_n+1}-1}
\overline F_\w(X_k))\right| >n\eps/3, \ \left|\frac {k_n}n-\mu_n(b_n)\right|\leq\eps,\ T_{b_n}<\eps n,\ T_{c_n}>n\right )\leq\\ \nonumber
&\tilde P^{\w,n}_0\left(\left|\sum_{k=1}^{k_n}
\xi_k\right| >n\eps/3, \ \left|\frac {k_n}n-\mu_n(b_n)\right|\leq\eps\right )\leq \\ \nonumber
&\tilde P^{\w,n}_0\left(\sup_{m=1,\ldots [n(\mu_n(b_n)+\eps)]}\left|\sum_{k=1}^{m}
\xi_k\right| >n\eps/3\right )\leq \frac {9(\mu_n(b_n)+\eps)\widetilde \Var_{b_n}^{\w,n}(\xi_1)}{n\eps^2}
\end{align}
Pluging in \eqref{al:devdecomp1} the bounds \eqref{al:knbound}, \eqref{eq:firsterm}, \eqref {al:lasterm} and \eqref {al:main} we obtain:
\begin{multline}\label{al:devdecomp2}
P^{\w}_0\left(\left| \S_n(F)-\Sigma_n(F)
\right|>\eps  \right)\leq
\\
P^{\w}_0\left(T_{b_n}\geq n\eps/2\right)+P^{\w}_0\left(T_{b_n}\geq n\tilde \eps\right)+2P^{\w}_0\left(T_{c_n}\leq n\right)+
\\
\frac{8}{n\eps^2}{\widetilde\Var^{\w,n}_{b_n}T_{b_n}}+\frac {9(1+\eps)\widetilde \Var_{b_n}^{\w,n}(\xi_1)}{n\eps^2}+\frac{2 c_n}{n\tilde\eps}.
\end{multline}
To conclude the proof we need to estimate $\widetilde\Var^{\w,n}_{b_n}(\xi_1).$
For $x\in Z_+$ 
introduce 
\begin{equation} \nn
Y_{x}=\sum_{j=0}^{T_{b_n}}\1_{\{x\}}(X_j)
\end{equation}
the local time in $x$ during $1$-th excursion from $b_n$ to $b_n.$ 
Note that under $\tilde P^{n}_{b_n},$
\begin{align}\label{al:xibound}
\xi_1 = \sum_{x=0,\ldots, c_n}  \overline F_\w(x) Y_x  \end{align}
and
$$ \widetilde\Var^{\w,n}_{b_n}(\xi_1)\leq (c_n+1)\|\overline F_\omega\|_\infty^2\sum_{x=0,\ldots, c_n}\widetilde\Var^{\w,n}_{b_n}(Y_x).$$
Taking  $\overline F_\omega=1$  in \eqref{al:xibound} we get:
\[\widetilde\Var^{\w,n}_{b_n}(T_{b_n}) \leq (c_n+1) \sum_{x=0,\ldots, c_n}\widetilde\Var^{\w,n}_{b_n}(Y_x).\]
Using Lemma \eqref{lem:variance}, which is given in Appendix,
for all $\eta >0,$ there exists $\delta >0$ and an event $\Omega_{\eta,\delta}\subset \Omega$ with
$\P(\Omega_{\eta,\delta})>1-\eta$ such that : $\forall \w\in \Omega_{\eta,\delta}, \forall x\in [0,c_n],$
$$\qquad \widetilde\Var^{\w,n}_{b_n}(Y_x)\leq n^{1-\delta}\;.$$
The proof of Lemma \ref{lem:variance} is given in appendix.
 As a consequence, for all $\eta >0,$ there exists $1>\delta >0,$ and a set $\Omega_{\eta,\delta}$ with $\P(\Omega_{\eta,\delta})>1-\eta,$ such that for all $\w\in \Omega_{\eta,\delta}$ it holds 
\begin{align}\label{al:devdecomp4}
P^{\w}_0\left(\left| \S_n(F)-\Sigma_n(F)
\right|>\eps  \right)\leq
C\frac{c_n^2n^{1-\delta}}{n}+2P^{\w}_0\left(T_{b_n}\geq (\eps +\tilde \eps)n\right)+2P^{\w}_0\left(T_{c_n}\leq n\right).
\end{align}
The last bound tends to zero. Indeed,
from \cite{Golosov}, Lemma 1, ${P^{\w}_{0}\left({T_{b_n}} >n\eps'  \right)\to 0}$ for all $\w\in\Omega$ and $\eps'>0.$
And from \cite {Golosov}, Lemma 7, for all $\w\in\Omega,$ ${P^{\w}_0\left(T(c_n)\leq n\right)\to 0}.$
\end{proof}

\subsection{Convergence in distribution  of  $ \Xi_n$}\label{sec:tight} 
Recall that $\Xi_n$ is a random element with values in $\ell^1$ given by
  \[
 \Xi_n:=\left \{\exp[-(V(b_n+x)-V(b_n))]\1_{\{-b_n,\ldots, c_n-b_n-1\}}(x)\ ;\ x\in\Z \right. \}
 \]
Denote by $P_{\Xi_n}$ the law of $ \Xi_n$
 on  $(\ell^1,     \BB(\ell^1)).$ 
The following proposition is the key technical result of the paper.
\begin{prop}\label{prop:main}
Suppose that Assumptions \eqref{as:rec} and\eqref{as:var} are satisfied. Then the following holds:
\begin{description}

\item {i)}
For all $\eta\in]0,1/2[$ and $\delta>0$,
\begin{equation*}
\lim_{K\to +\infty}\liminf_{n\to\infty}\P\left ( V(b_n+x)-V(b_n) \geq \delta x^{\eta},\ \forall x \in \llbracket K, c_n-b_n\rrbracket \  \ \right )=1.
\end{equation*}
\item {ii)} Suppose that in addition that  Assumptions \eqref{as:aritm} is satisfied.
Then for all $\eta\in]0; 1/3[$,
\begin{equation*}
\lim_{K\to +\infty}\liminf_{n\to\infty}\P\left ( V(b_n-x)-V(b_n) \geq  x^{\eta},\ \forall x \in  \llbracket K, b_n\rrbracket\ \ \right )=1.
\end{equation*}
\end{description}
\end{prop}
The proposition \eqref{prop:main} is proven in Section \eqref{sec:propmain}.
\begin{prop}\label{prop:tight}
Suppose that Assumptions \eqref{as:rec},\eqref{as:var} and \eqref{as:aritm} are satisfied.
Then the sequence $P_{\Xi_n}$ is relatively compact
 in $(\ell^1, \|\cdot\|_1)$.

\begin{proof}
Recall that $(\ell^1,\|.\|_1)$ is a complete separable metric space and hence , using Prohorov's theorem (\cite {Bil}),
 the sequence of distributions $P_{\Xi_n}$ is relatively compact if and only if it is tight.
 Recall also the characterization of the compacts in $(\ell^1,\|.\|_1):$
  $$\KK\subset \ell^1\quad\mbox{is compact}\quad\Longleftrightarrow\quad\sup_{l\in\KK}\|l\|_1<\infty\quad\mbox {and}\quad \lim_{N\to\infty}\sup_{l\in\KK} \sum_{|x|\geq N}|l_x|=0.$$
%
 Let 
 $\eta \in (0,1/3),$ $K>0,$ 
 and denote
  \[ \KK(\eta, K):=\{l\in\ell^1;\;   |l_x|\leq 1\quad  \mbox{and}\quad  \forall x\in \Z  \quad   |l_x|\leq e^{- |x|^{\eta}}\quad \mbox{if}\quad  |x|\geq K\}.\]
 Since $\sum_{x\in\Z} e^{- |x|^{\eta}}<\infty,\ $  $\KK(\eta, K)\ $ is a compact in $\ell^1.$
As a consequence of Proposition \ref{prop:main}, for any fixed $\eta\in (0,1/3),$  forall $\eps >0$,  there exists 
$K>0,$   such that 
\[\liminf_{n}\P\left ( e^{-(V(b_n+x)-V(b_n))}<e^{- |x|^{\eta}},\ \forall  x\in\llbracket -b_n, c_n-b_n\rrbracket; \quad |x-b_n|>K \ \right )\geq 1-\eps.\]
Hence, for $n$ large enough
$ 
 \P(\Xi_n\in\KK({\eta, K)})\geq 1-\eps$
 and  the sequence $P_{\Xi_n}$ is tight.
 \end{proof}
 \begin{theo}\label{th: weakconv}
 Suppose that Assumptions \eqref{as:rec},\eqref{as:var} and \eqref{as:aritm} are satisfied.
 Then the sequence $P_{\Xi_n}$ 
converges weakly to the law of  
$\ \{\exp[-\widetilde V(x)]\ ;\ x\in\Z\}$ 
 on $(\ell^1, \BB(\ell^1))$.
\end{theo} 
 \begin{proof}
 From \cite{Golosov} the following convergence of finite dimensional distributions  (fidi) holds: 
$$\big\{\exp[-(V(b_n+x)-V(b_n))]\ ;\ x\in\Z \big \}\stackrel{\rm fidi}{\longrightarrow}
 \big\{\exp[-\widetilde V(x)]\ ;\ x\in\Z\big\}.$$
Using $b_n\to\infty$ a.s. and $c_n-b_n\to\infty$ a.s.,
the finite dimensional distributions of  $\Xi_n$ converge weakly to those of $\ \{\exp[-\widetilde V(x)]\ ;\ x\in\Z\}$ . 
 Denote
 by $\M$ the class of continuous, bounded, finite-dimensional functions:
  $$\M:=\bigcup_{k\in\N^*}
  \big\{ f\in\C_b(\ell^1):  f(l_i)=f(l_i') \; \forall i\in \llbracket-k,k \rrbracket \implies f(l)=f(l') \big\}.$$
 It is clear that $\M$ separates points: if $l\in\ell^1,$ $l'\in\ell^1,$ $l\neq l',$ that there exists $f\in\M,$ such that $f(l)\neq f(l').$ 
 Since $(\ell^1,d)$ is separable and complete, and $\M$ separates points, using Theorem 4.5 from \cite{EK} $\M$ is separating.
 Now the claim follows directly from  proposition \eqref{prop:tight} and lemma 4.3 of \cite{EK}.
 
 \end{proof}
 
\section{Proof of Proposition \eqref{prop:main}}\label{sec:propmain}
\begin{proof}
We start by  proving $ii).$
Let $T_0:=0$ and for all $\ell\in\N^*,$ put
$$\ T_{\ell+1}:=\inf\{y>T_{\ell},\ V(y)<V(T_{\ell})\}.$$
The sequence $(T_{\ell})_{\ell\in\N}$ is the sequence of the strict descending ladder epochs of $V.$
Let $$e_{\ell}=((V(z)-V(T_{\ell-1})),\; T_{\ell-1}\leq z<{T_{\ell}}),\quad \ell\in\N^*.$$
Using the strong Markov property of $V,$ the sequence $(e_{\ell});\; \ell\in\N^*$  is an i.i.d. sequence.   Let $N(n)$ be a random time, such that $b_n=T_{N(n)}.$ Namely, setting as previously $L_n:=\ln n+\sqrt {\ln n},$ we have following \cite{Golosov} p.492, 
$$N(n):=\inf\{\ell\in\N^*;\; \max\{V(z)-V(T_{\ell-1});\; T_{\ell-1}\leq z<T_{\ell}\}\geq L_n\}.$$
 Due to the independence and the equidistribution of the excursions $(e_{\ell}); \; \ell\in\N^*,$ the random variable  $N(n)$ is geometrically distributed with the parameter  
 \begin{equation}\label{param}
 p_n:=\P(\tau_{L_n}<\tau_{0-}),\;\;  \mbox{where}
 \end{equation}
 $$\tau_{L_n}:=\inf\{z>0,\; V(z)\geq L_n\}\quad\mbox{and}\quad\tau_{0-}:=\inf\{z>0;\; V(z)<0\}.$$ 
 Let $K\in\N^*$ and denote 
\begin{align*}
C(\eta,K,n):&=\{\forall x=K,\ldots, b_n,\; \ V(b_n-x)-V(b_n)\geq x^{\eta}\}\\
&=\{\forall y=0,\ldots, b_n-K,\;  V(y)-V(b_n)\geq |y-b_n|^{\eta}\}.
\end{align*}
Keeping in mind the relation $b_n=T_{N(n)},$ we can observe that
\begin{equation}\label{eq:Ceta}
C(\eta, K,n)=
\{\forall \ell =0,\ldots N(n)-1;\  |T_{\ell}-T_{N(n)}|\geq K;\ V(T_{\ell})-V(T_{N(n)})\geq|T_{\ell}-T_{N(n)}|^{\eta}\}
\end{equation}
Indeed, let $\ell(y)\in \N$ be the number of ladder excursion containing $y$, i.e.
$T_{\ell(y)}\leq y<T_{\ell(y)+1}$, then using the fact that
the function $x\to x^{\eta}$ is increasing on $\R_+,$ $V(y)\geq V(T_{\ell(y)})$ and $V(T_{\ell})-V(T_{N(n)})\geq|T_{\ell}-T_{N(n)}|^{\eta} $  for all $\ell\in\N,$ we have
%
 $$V(y)-V(T_{N(n)})=V(y)-V(T_{\ell(y)})+V(T_{\ell(y)})-V(T_{N(n)})\geq |T_{\ell(y)}-T_{N(n)}|^{\eta}\geq |y-T_{N(n)}|^{\eta}.$$
%
  Which prove \eqref{eq:Ceta}.
  
   Due to the arithmeticity of the law of $\log\rho_0,$ for all $\ell\in\N,$  $V(T_{\ell})-V(T_{\ell+1})\geq h.$ 
Therefore we can write:
\begin{align}\label{al:deco1}
&\P(C^c(\eta,K,n))=\\ \nonumber
&\sum_{N=1}^{\infty}\P(\exists \ell=0,\ldots, N-1,\ T_N-T_{\ell}\geq K,\; V(T_{\ell})-V(T_{N})<|T_{\ell}-T_{N}|^{\eta};\  N(n)=N)\leq\\\nonumber
&\sum_{N=1}^{M}\P(N(n)=N)+\sum_{N=M+1}^{\infty}\P(\exists m=1,\ldots, M,\ T_N-T_{N-m}\geq K ;\  N(n)=N)+\\ \nonumber
&\sum_{N=M+1}^{\infty}\P(\exists m=M+1,\ldots, N,\ T_N-T_{N-m}\geq  (hm)^{1/\eta};\  N(n)=N):=\\ \nonumber
&S_1 ( n,M)+S_2(K,n,M)+S_3(\eta, n,M).\nonumber
\end{align}
Here in the third line we denoted
$m=N-\ell$  the number of "ladder" between $T_N$ and $T_{\ell}$ and the auxiliary $M\in\N^*$ will be choose later.
We obviously have
\begin{equation}\label{eq:sn1}
S_1 (n, M)=\sum_{N=1}^M(1-p_n)^{N-1}p_n=1-(1-p_n)^M\sim M p_n\to 0\quad\mbox {if}\quad n\to\infty.
\end{equation}
For the second sum we can write, denoting $\sigma_{\ell}:=T_{\ell}-T_{\ell-1}$ the length of the $\ell$-th ladder 
\begin{align}\label{al:s2}
&S_2(K,n,M):= 
\sum_{N=M+1}^{\infty}\P(\exists m=1,\ldots, M,\ T_N-T_{N-m}\geq K ;\  N(n)=N)\leq\nonumber\\
&\sum_{N=M+1}^{\infty}\P(T_N-T_{N-M}\geq K;\; N(n)=N)=
\sum_{N=M+1}^{\infty}\P(\sigma_{N-M+1}+\ldots +\sigma_N\geq K;\; N(n)=N)\leq\nonumber\\
&\sum_{N=M+1}^{\infty}\sum_{\ell=N-M+1}^N\P(\sigma_{\ell}\geq  K/ M;\; N(n)=N).\nonumber
\end{align}
 The event $\{N(n)=N\}$ can be written as
$$\{N(n)=N\}=\{\tau^1_{0-}<\tau^1_{L_n};\ldots; \tau^{N-1}_{0-}<\tau^{N-1}_{L_n}; \tau^{N}_{0-}>\tau^{N}_{L_n}\},$$
where we denoted 
 $\tau^{\l}_{0-}:=\tau_{0-}\circ\theta_{ T_{\l}}$ and $\tau^{\l}_{L_n}:=\tau_{L_n}\circ\theta_{ T_{\l}}.$ Let $c_{Sp}>0$ be such that $\forall a>c_{Sp},$ $\P(\tau_{0-}>a)\leq {C}{ a^{-1/2}},$ where $C$ is a positive constant.  Following \cite{Spitzer} we can choose such a constant $c_{Sp}$.
Let $M$ in \eqref{al:deco1} be fixed in a such a way that $K/M>c_{Sp}$. 
Then, using the independence of the ladder excursions, together with the definition \ref {param}, we can write (note that $\sigma_1=\tau^1=\tau_{0-}$),

\begin{equation*}
\P(\sigma_{\ell}\geq \frac KM;\; N(n)=N) \leq \P(\sigma_{\ell}\geq \frac K{M}\cap \tau^{\ell}_{0-}<\tau^{\ell}_{L_n})p_n(1-p_n)^{N-2}\leq
C\sqrt{\frac MK} p_n(1-p_n)^{N-2}. \end{equation*}

Using this bound we see that 
\begin{equation}
\label{S2bound}
S_2(K, n,M)\leq\sum_{N=M+1}^{\infty}\frac { C(M_1)^{3/2}p_n(1-p_n)^{N-2}}{\sqrt K}= \frac { C(M_1)^{3/2}(1-p_n)^{M-1}}{\sqrt K}\leq \frac { C(M_1)^{3/2}}{\sqrt K} .
\end{equation}
To find a bound for $S_3(\eta,n,M)$ we introduce

$$B_m:=\{ T_N-T_{N-m}\geq  (hm)^{1/\eta},\;\;  T_N-T_{N-(m-1)}<  (h(m-1))^{1/\eta}  \}.$$
Then
\begin{align}\label{al:S3}
&S_3(\eta,n,M)=\\ \nonumber 
&\sum_{N=M+1}^{\infty}\P(\exists m=M+1,\ldots, N,\ T_N-T_{N-m}\geq  (V(T_{N-m})-V(T_n))^{1/\eta};\  N(n)=N)\leq \\ \nonumber 
&\sum_{N=M+1}^{\infty}\P(\exists m=M+1,\ldots, N,\ T_N-T_{N-m}\geq  (hm)^{1/\eta};\  N(n)=N) \leq \\ \nonumber 
&\sum_{N=M+1}^{\infty}\sum_{m=M+1}^N\P(B_m\cap N(n)=N)
\end{align}
Remark that $T_N-T_{N-m}=T_N-T_{N-m+1}+\sigma_{N-m}$ and hence  $B_m\subset \{\sigma_{N-m}\geq c'' m^{\frac 1{\eta}-1}\}$ for $c=\eta^{-1}\sup_{[1;2]}x^{1/\eta-1}\leq \eta^{-1}2^{1/\eta-1}.$ Chose $M$ such that $c'' M^{\frac 1{\eta}-1}>c_{Sp}$ where $c_{Sp}$ is again a constant of Spitzer. Then we can write

\begin{multline}
\sum_{m=M+1}^N\P(B_m\cap  N(n)=N )\leq \sum_{m=M+1}^N\P(\sigma_{N-m}\geq c'' m^{\frac 1{\eta}-1}; \tau^m_{0-}<\tau^m_{L_n})p_n(1-p_n)^{N-2}\\
\leq p_n(1-p_n)^{N-2}\sum_{m=M+1}^N\P(\sigma_{N-m}\geq c'' m^{\frac 1{\eta}-1})=
p_n(1-p_n)^{N-2}\sum_{m=M+1}^N\P(\tau_{0-}\geq c'' m^{\frac 1{\eta}-1})\\
\leq p_n(1-p_n)^{N-2}\sum_{m=M_1}^N\frac {C}{
m^{\frac {(1-\eta)}{2\eta} }       }\leq Cp_n(1-p_n)^{N-2}M^{3/2-\frac 1{2\eta}   }
\end{multline}
Finally  we get
\begin{equation} \label{SS3}
S_3(n)\leq \sum_{N=M}^{\infty}p_n(1-p_n)^{N-2} M^{3/2-\frac 1{2\eta}   }\leq M^{3/2-\frac 1{2\eta}}.
\end{equation}

And finally  putting together \eqref {eq:sn1}, \eqref {S2bound} and \eqref{SS3} we obtain from \eqref{al:deco1}:
\begin{align}\label{al:deco11}
&\P(C^c(\eta,K,n))=\\ \nonumber
&S_1 ( n,M)+S_2(K,n,M)+S_3(\eta, n,M)\leq 
 Mp_n+\frac {M^{3/2}}{\sqrt K} + M^{3/2-\frac 1{2\eta}}\nonumber.
\end{align}
If  $\eta <1/3,$ then $3/2-\frac {1}{2\eta}<0.$ Remember that the bound on $S_2$ is valuable if $K/M>c_{Sp}$ and that on $S_3$, if $c"M^{\frac {1-\eta}{\eta}}>c_{Sp}.$ Hence we first choose $M$ large enough, such that 
simultaneously $c"M^{\frac {1-\eta}{\eta}}>c_{Sp}$ and $M^{3/2-\frac 1{2\eta}}\leq \eps.$ Then we  get
$$\lim_{K\to\infty}\limsup_{n\to\infty}\P(C^c(\eta, K, n))\leq \eps$$
which, sins $\eps>0$ is arbitrary, concludes the proof.

\ec
Now we prove  $i)$.\\
Recall $V(x)= \sum_{y=1}^x \log \rho_y$  if $x>0$, $(\log \rho_y)_{y\in\Z} $ i.i.d. centered.  Denote 
\begin{eqnarray}\nn
\tau_{\ln}=\inf \big\{x\in\Z_+,\ V(x)\geq\log n+\sqrt{\log n} \big\}\;,\quad \tau_{0-}= \inf \big\{x\in\Z_+,\ V(x)<0 \big\}\;.
\end{eqnarray}
Using strong Markov property,
$$\cL(\{V(b_n+x)-V(b_n),\ x=K-b_n,\ldots, c_n-b_n\})=\cL(\{V(x), \ x=K,\ldots, \tau_{\ln} | \tau_{\ln}<\tau_{0-}\})\;.$$
Hence we have to prove
\begin{equation}\label{eq:CN'}
\lim_{K\to +\infty}\inf_{n}\P\left ( V(x)\geq\delta x^{\eta},\ \forall x =K,\ldots,  \tau_{\ln}\wedge \tau_{0-} \ | \tau_{\ln}<\tau_{0-} \right )=1.
\end{equation}

or equivalently
\begin{equation}\label{eq:amontrer}
\lim_{K\to +\infty}\sup_{n}\frac {\P\left ( \exists x \geq K,\  x< \tau_{\ln}\wedge \tau_{0-},\    V(x)<\delta x^{\eta},\  \tau_{\ln}<\tau_{0-}\right )}{ \P \left (\tau_{\ln}<\tau_{0-}\right )} =0.
\end{equation}

%
%
The following proof is inspired by \cite{riter}.
Let $c>1$ be an integer such that $\forall a>c,$ $\P(\tau_{0-}>a)\leq {C}{ a^{-1/2}},$ where $C$ is a positive constant.  Following Spitzer \cite{Spitzer} we can choose such a constant $c$.
 For $r\in \N^*,$ denote ${\cal C}_r$ the following event
\begin{equation*}
{\cal C}_r:=\{\exists x\in[c^{r-1},c^{r}[; \;\;  x<\tau_{\ln}\wedge \tau_{0-};\;\;    V(x)<\delta x^{\eta};\;\;  \tau_{\ln}<\tau_{0-}\}
\end{equation*}
and denote for $x\in [c^{r-1},c^{r}[$
\begin{equation*}
{\cal A}_x:=\{ \forall z\in[c^{r-1}, x[; \;\; V_z\geq \delta z^{\eta};\;\;  V(x)<\delta x^{\eta};\;\;   x<\tau_n\wedge \tau_{0-}\}
\end{equation*}
Note that ${\cal A}_x$ are disjoint for $x\in [c^{r-1},c^{r}[$  and that

$${\cal C}_r:= \bigcup_{x=c^{r-1}}^{c^{r}-1}\left \{ {\cal A}_x \cap \{ \tau_{\ln}<\tau_{0-} \} \right\}.$$

For all $x\in\N,$ denote $\P:=\P_0,$  $\FF_x=\sigma\{V_0,\ldots, V_x\}$. Denote $L_n=\log n+\sqrt{\log n}.$ 
It is easy to see that
$$\forall 0\leq y<L_n,\quad \P_y (\tau_{\ln}<\tau_{0-})\leq \frac{y+c_0}{L_n}$$
for some positive constant $c_0=\sup_{n,y\geq 0} \E_y(-V(\tau_{0-})\mid\tau_{0-}< \tau_{\ln})$. 
Indeed, using Doob stopping theorem, and the fact that $V_{\tau_L}>L,$
$$ \P_y (\tau_{\ln}<\tau_{0-})=\frac {y-\E_y [V_{\tau_{0-}}\mid\tau_{0-}< \tau_{\ln}]}{\E_y[ V_{\tau_L}\mid\tau_{L}< \tau_{0-}]-\E_y [V_{\tau_{0-}}\mid\tau_{0-}< \tau_{\ln}]}.$$

Since the event $\{\tau_{\ln} < \tau_{0-}\}$ is invariant under shift, using Markov property we can write:
\begin{multline*}
\P_0\left ({\cal A}_x\cap \{ \tau_{\ln}<\tau_{0-} \}\right )= \E_0 \left [\P_0\left ({\cal A}_x\cap \{ \tau_{\ln}<\tau_{0-}\}|\FF_x \right )\right ]=\\
\E_0\left ( \forall z\in[c^{r-1}, x[; \;\; V_z\geq \delta z^{\eta};\;\;  V(x)<\delta x^{\eta};\;\;   x<\tau_{\ln}\wedge \tau_{0-};\;\; \P_{V_x}(\tau_{\ln}<\tau_{0-})\right )\leq\\
\P_0\left ( \forall z\in[c^{r-1}, x[; \;\; V_z\geq\delta z^{\eta};\;\;  V(x)<\delta x^{\eta};\;\;   x<\tau_{\ln}\wedge \tau_{0-}\right ){(\delta x^{\eta}+c_0)}(L_n)^{-1}=\\
\P({\cal A}_x){(\delta x^{\eta}+c_0)}(L_n)^{-1}.
\end{multline*}
Hence, using the fact that ${\cal A}_x$ are disjoint and that   
$$\forall x=c^{r-1},\ldots, c^{r}-1;\quad {\cal A}_x\subset \{\tau_{0-}>x\}\subset \{\tau_{0-}>c^{r-1}\},$$
with our choice of $c$, we have
\begin{multline*}
\P({\cal C}_r)= 
\sum_{x=c^{r-1}}^{c^{r}-1} \P\left  ({\cal A}_x \cap \{ \tau_{\ln}<\tau_{0-} \} \right )
\leq \frac{\delta c^{r\eta}+c_0}{L_n} \sum_{x=c^{r-1}}^{c^{r}-1} \P \left ( {\cal A}_x \right )=\\
\frac{\delta c^{r\eta}+c_0}{L_n}\P \left ( \bigcup_{x=c^{r-1}}^{c^{r}-1}{\cal A}_x \right )\leq
\frac{\delta c^{r\eta}+c_0}{L_n}\P \left ( \tau_{0-}>c^{r-1}\right )\leq\\
C\sqrt c(\delta c^{r(\eta-1/2)}+ c_0 c^{-r/2})(L_n)^{-1}.
\end{multline*}
Finally, for any $n$ and $R\geq 2$,
\begin{multline*}
\P\left ( \exists x >c^{R-1};\;\;   x< \tau_{\ln}\wedge \tau_{0-};\;\;    V(x)<\delta x^{\eta};\; \;  \tau_{\ln}<\tau_{0-}\right )\leq\\ 
\sum_{r=R}^{\infty}\P({\cal C}_r)\leq C_1(\delta + c_0 c^{-R\eta}) c^{R(\eta-1/2)} (L_n)^{-1},
\end{multline*}
and \eqref{eq:amontrer} follows from $\P_0 ( \tau_{\ln}<\tau_{0-})\sim(L_n)^{-1}$.

\end{proof}

\end{prop}

\section{Proof of Proposition \ref{thm:inv}}\label{sec:inv}


(i) By definition of $\S_n$, we have for $F \in {\mathcal C}( \Omega)$, 
\begin{equation}\nn
\int_\Omega RF d \S_n = \frac 1n \sum_{k=1}^n RF (\overline {\omega}_k ) \;,
\end{equation}
 and by definition of $R$, 
 $$RF (\overline {\omega}_k ) = E_0^\omega [F(\overline {\omega}_{k+1} ) | X_k]\;.$$
 $$\EE[RF(\bar\w_k)]=\EE[F(\bar\w_{k+1})].$$
Thus,
$$
{\mathbf E} \int_\Omega RF d \S_n  =
{\mathbf E} \int_\Omega F d \S_n  + {\mathcal O}\big(\frac{\|F\|_\8}{n}\big) \,.
$$
 Note that from the definition of the topology  of weak convergence on ${\mathcal P}(\Omega),$ for all continuous bounded $H:\Omega\rightarrow\R,$ the function
 ${\mathcal H}: {\mathcal P}(\Omega)\rightarrow \R$ given for all $\mu\in {\mathcal P}(\Omega)$by
 ${\mathcal H}(\mu):=\int_{\Omega}Hd\mu$ is continuous.

Hence, as $n \to \8$, we obtain from the
 convergence in distribution given by the Theorem \ref {thm:main},

$${\mathcal E} \int_\Omega RF d \S_\8  =
{\mathcal E} \int_\Omega F d \S_\8   \,,
$$
that is, $\int_\Omega RF d \IS = \int_\Omega F d \IS$. Hence $\IS$ is invariant.  

(ii) To show reversibility, we need to prove that for the chain starting from $\overline{\omega}_0 \sim \IS$, 
\begin{equation}\label{eq:9/11/1}
\EE F( \overline{\omega}_0) G( \overline{\omega}_1)  = \EE F( \overline{\omega}_1) G( \overline{\omega}_0) 
\end{equation}
for $F, G$ continuous and bounded on $\Omega$. Let $\widetilde \P$ denote the law of $ \widetilde \omega$. 
By definition of the transition $R$, the LHS in \eqref{eq:9/11/1} is equal to
\begin{eqnarray} 
\begin{split}
\nn
&\int_\Omega d \widetilde \P(\widetilde \omega)  \sum_{x \in \Z} \widetilde \nu (x) F(T_x \widetilde\omega) 
\big[ \widetilde \omega(x) G(T_{x+1} \widetilde \omega) + (1-\widetilde \omega(x)) G(T_{x-1} \widetilde \omega)  \big]\\ \nn
&= 
\int_\Omega d \widetilde \P(\widetilde \omega) 
  \big[ \sum_{x \in \Z} \widetilde \nu (x) F(T_x \widetilde \omega) \widetilde \omega(x) G(T_{x+1} \widetilde \omega) 
+
 \sum_{x \in \Z} \widetilde \nu (x) F(T_x  \widetilde  \omega)  (1-\widetilde \omega(x)) G(T_{x-1} \widetilde \omega)  \big]\\ \nn
 & = 
 \int_\Omega d \widetilde \P(\widetilde \omega) 
  \big[ \sum_{x \in \Z} \widetilde \nu (x+1) F(T_x  \widetilde  \omega)(1- \widetilde \omega({x+1})) G(T_{x+1} \widetilde \omega) 
  +
  \sum_{x \in \Z} \widetilde \nu (x-1) F(T_x  \widetilde \omega)  \widetilde \omega({x-1}) G(T_{x-1} \widetilde \omega)  \big] \\\nn
 & = 
 \int_\Omega d \widetilde \P(\widetilde \omega) 
  \big[ \sum_{y \in \Z} \widetilde \nu (y) F(T_{y-1} \widetilde  \omega)(1- \widetilde \omega({y})) G(T_{y} \widetilde \omega)  
  +
   \sum_{z \in \Z} \widetilde \nu (z) F(T_{z+1}  \widetilde  \omega)  \widetilde \omega({z}) G(T_{z} \widetilde \omega)  \big]
 \end{split}
\end{eqnarray}
using
the relation
\begin{equation}\label{eq:rel}
\tilde \nu(x-1)\tilde \w(x-1)=\tilde\nu(x)(1-\tilde\w(x))=\exp(-\tilde V(x))/2\sum_{y\in\Z}\exp(-\tilde V(y))
\end{equation}

 in the third line and change of variables $y=x+1, z=x-1$ in the last one. The last expression being the RHS of  \eqref{eq:9/11/1},
we obtain reversibility of $\IS$, which implies invariance by taking $g=1$.

(iii) Consider the event 
$$
\Omega_+= \left\{ \omega \in \Omega : \sum_{y=1}^x \log \frac{1-\omega(y)}{\omega(y)} \geq 0\quad \forall x \geq 1 \right\}\;.
$$
By construction of $\widetilde V$ and since $V$ is a mean-zero random walk under $\P$, we have
$$
\IS ( \Omega_+) = 1\,, \quad  \P ( \Omega_+)=0\;,
$$
so the two measures are mutually singular.
%
\qed 
\begin{rema}
From the proof (ii) of reversibility we see that there exist many invariant measures by the transition $R$. Indeed, the proof works thanks to the relation \eqref{eq:rel}, (which means the reversibility of the measure $\tilde \nu$ with respect to  the walk in the environment $\tilde\w$)  and thanks to the fact that $\sum_{x \in \Z} \exp[-\tilde V(x)]$ finite a.s. . Since for any $\widehat V$ (in the place  of $\tilde V$), $\hat\w $ and $\hat \nu$ defined with the help of
$\widehat V$ using \eqref{equa:tildeOmega},\eqref{equa:Nu} the measure on $\Omega$ given by  $E \sum_{x \in \Z} \widehat \nu (x) \delta_{T_x \widehat{\w}},$ where the expectation is taken w.r.t. to the law of $\widehat V$,
is reversible for $R.$
\end{rema}

\section{Appendix}
%
Recall \begin{equation} \nn
Y_{x}=\sum_{j=0}^{T_{b_n}}\1_{\{x\}}(X_j)
\end{equation}

\begin{lemm}\label{lem:variance}
For all $\eta >0$ there exists $\delta >0$ and an event $\Omega_{\eta,\delta}\subset \Omega$ with
$$\P(\Omega_{\eta,\delta})>1-\eta$$ such that for all $\w\in \Omega_{\eta,\delta},$  for all $x\in [0,c_n],$
$$\qquad \widetilde\Var^{\w,n}_{b_n}(Y_x)\leq n^{1-\delta}\;.$$
\end{lemm}

\begin{proof} 
Note that $Y_{b_n}=1$.
Using $(2.10)$ of \cite{GPS}, for $x\neq b_n$ 
we have
$\widetilde\Var^{\w}(Y_{x})\leq \frac 4{\beta(x)},$
where
$$\beta(x)=(1-\w_{x}) \tilde P^{\w,n}_{x-1}\left (T(b_n)<T(x)\right ), \quad x=b_n+1,\ldots c_n;$$
$$\beta(x)=\w_{x} \tilde P^{\w,n}_{x+1}\left (T(b_n)<T(x)\right ), \quad x=0,\ldots, b_n-1.$$

 Then using the hypoellipticity, $(2.8)$ and   $(2.11)$ of \cite{GPS} we can find a constant $C_0$ such that for $y=1,\ldots c_n-b_n,$
\begin{equation}
\label{eq:variance}
\widetilde\Var^{\w}(Y_{b_n+y})\leq C_0 \sum_{j=b_n}^{b_n+y-1}e^{V(j)-V(b_n+y-1)},
\end{equation}
and  for $y=-b_n,\ldots, -1,$
\begin{equation}
\label{eq:variance-}
\widetilde\Var^{\w}(Y_{b_n+y})\leq  C_0 \sum_{j=b_n+y}^{b_n-1}e^{V(j)-V(b_n+y)}.
\end{equation}

However the exponential is missing in the bound $(2.13)$ of \cite{GPS}, hence, to complete the proof  we have to  handle the term
\begin{equation}
A_n^{\w}(x):=
\left\lbrace
\begin{array}{ccc}
\sum_{j=b_n}^{b_n+x-1}e^{V(j)-V(b_n+x-1)} & \mbox{si} & x=1,\ldots c_n-b_n,\\
\sum_{j=b_n+y}^{b_n-1}e^{V(j)-V(b_n+y)} & \mbox{si} & x=-b_n,\ldots, -1
\end{array}\right.
\end{equation}
more carefully. Using Komlos-Mayor-Tusnady theorem we can construct a probability space $(\Omega,{\cal A},P)$ on which are defined the environment $\w$ and a  Brownian motion $W$ s.t. a.s.
\[\sup_{0\leq s\leq t}|V(s)-\sigma W_s|\leq C\ln t,\]
where $\sigma^2=\E(\log\rho)^2$ and $V(s)$, $s\geq 0,$ is defined equal to $V(j)$ on $[j,j+1[$, $j\in \Z_+.$
Let $t=c_n,$ $W^{(n)}(s)=\frac 1{\sigma\ln n}W(s\sigma^2\ln^2n),$ $\bar b_n=\frac{b_n}{\sigma^2\ln^2n},$ $\bar c_n=\frac{c_n}{\sigma^2\ln^2n}$ and $\ln_2:=\ln\ln.$ Then we can write, with $z=\sigma^2\ln^2n\times u,$
\begin{align}
&A_n^{\w}(x)=\int_{b_n}^{b_n+x-1}\exp\{\sigma|W(z)-W(b_n+x-1)|+O(\ln c_n)\}dz=\\ \nonumber
&=\sigma^2\ln^2n\int_{\bar b_n}^{\bar b_n+\frac{x-1}{\sigma^2\ln^2n}}e^{\{\sigma^2\ln n[W^{(n)}(u)-W^{(n)}(\bar b_n+\frac{x-1}{\sigma^2\ln^2n})]\}}du\times e^{(O(\ln c_n))}.
\end{align}
Denote 
\[\Delta_n:=\max\{W^{(n)}(u)-W^{(n)}(v);\ \bar b_n\leq u\leq v\leq \bar c_n \}\]
and
\[\Delta'_n:=\max\{W^{(n)}(u)-W^{(n)}(v);\ 0\leq v\leq u\leq \bar b_n \}\]

It follows that
\[\max\{A_n^{\w}(x),\ x=0,\ldots, c_n-b_n\}\leq c_n\exp\{\sigma^2\ln n\times\Delta_n+O(\ln c_n)\}.\]
Similarly, for $x\in [-b_n,\ldots, 0],$

\[\max\{A_n^{\w}(x),\ x=-b_n,\ldots, 0\}\leq c_n\exp\{\sigma^2\ln n\times\Delta'_n+O(\ln c_n)\}.\]
And hence, 
\begin{equation}\label{eq:estim}
\max \{\frac{\ln A_n^{\w}(x)}{\sigma^2\ln n};\ x\in[-b_n;c_n-b_n]\}\leq \Delta_n\vee\Delta'_n+O(\frac{\ln c_n}{\ln n}).
\end{equation}
By Donsker theorem, $(\bar b_n,\bar c_n,W^{(n)})$ converges in distribution to $(\bar b,\bar c, \bar W)$, where $\bar W$ is a brownian motion, $$\bar c=\inf \{s\geq 0,\ \bar W(s)-\min_{0\leq t\leq s}\bar W(t)\geq 1\}$$ and $$\bar b:=\inf\{u\geq 0, \ \bar W(u)=\min_{0\leq t\leq \bar c}\bar W(t)\}.$$
Therefore, $(\Delta_n,\Delta'_n)$ converges in distribution to  $(\Delta,\Delta'),$ with
\[\Delta:=\max\{\bar W(u)-\bar W(v);\ \bar b\leq u\leq v\leq \bar c\}\]
and
\[\Delta':=\max\{\bar W(u)-\bar W(v);\ 0\leq v\leq u\leq \bar b \}\]
Both $\Delta<1$ and $\Delta'< 1$ a.s., so
$\Delta\vee\Delta'< 1$ a.s. 
Using the monotonicity  and still Donsker theorem it follows that $\forall \eta >0,\ \exists \delta>0,$ such that 
\begin{equation}\label{eq:estdelta}
\liminf_{n\to\infty}\P(\Delta_n\vee\Delta_n^{'}<1-\delta)>1-\eta
\end{equation}
It follows from \eqref{eq:estdelta} and \eqref{eq:estim} that $\forall \eta>0,\exists \delta >0,$ s.t.
\begin{equation}\label{eq:final}
\liminf_{n\to\infty}\P(\max_{[-b_n,c_n-b_n]} A_n^{\w}(x)\leq n^{1-\delta})\geq 1-\eta
\end{equation}
Denote $\Omega_{\eta,\delta}:=\{\w\in\Omega;\ \max_{[-b_n,c_n-b_n]} A_n^{\w}(x)\leq n^{1-\delta}\}$ such that 
$\P(\Omega_{\eta,\delta})>1-\eta.$ 
Suppose that $\w\in \Omega_{\eta,\delta}$. Then for all $x\in [-b_n,c_n-b_n],$ $\widetilde\Var_{\w} Y_n(x)\leq n^{1-\delta}$.
\end{proof}

\bibliography{Dev_31_08_22}
\bibliographystyle{chicago}
\end{document}